\documentclass{article}
\usepackage[a4paper, total={7in, 10in}]{geometry}
\usepackage[T1]{fontenc}
\usepackage{graphicx}
\usepackage{mathtools}
\usepackage{amsmath,amsthm,amssymb,mathrsfs,amstext, titlesec,enumitem, comment, graphicx, color, xcolor, stmaryrd,mathabx}
\usepackage{thmtools}
\usepackage{xpatch}
\usepackage{tikz-cd}
\usepackage{nameref}
\usepackage{hyperref}
\makeatletter
\def\Ddots{\mathinner{\mkern1mu\raise\p@
\vbox{\kern7\p@\hbox{.}}\mkern2mu
\raise4\p@\hbox{.}\mkern2mu\raise7\p@\hbox{.}\mkern1mu}}
\makeatother
\newtheorem{theorem}{Theorem}[section]
\newtheorem{corollary}[theorem]{Corollary}
\newtheorem{lemma}[theorem]{Lemma}

\theoremstyle{definition}
\newtheorem{definition}[theorem]{Definition}
\newtheorem{ex}[theorem]{Examples}

\begin{document}
	
	\title{ Dynamical characterization of central sets\\
 in adequate partial semigroups  }

	\date{}
	\author{Pintu Debnath
		\footnote{Department of Mathematics,
			Basirhat College,
			Basirhat-743412, North 24th parganas, West Bengal, India.\hfill\break
			{\tt pintumath1989@gmail.com}}
   \and
		Sayan Goswami 
		\footnote{Department  Mathematics,
			Ramakrishna Mission Vivekananda Educational and Research Institute,
			Belur, Howrah, 711202, India\hfill\break
			{\tt sayan92m@gmail.com}}
		\and
	Sourav Kanti Patra
	\footnote{Department of Mathematics,
		GITAM University Hyderabad Campus,
		Hyderabad, Telangana- 502329, India. \hfill\break
		{\tt spatra3@gitam.edu}	
		}
	}
	\maketitle	
	
	\begin{abstract}
Using the methods from topological dynamics,  H. Furstenberg introduced the notions of Central sets and proved the famous Central Sets Theorem which is the simultaneous extension of the van der Waerden and Hindman Theorem. Later N. Hindman and V. Bergelson found an equivalent formulation of Central sets in the set of natural numbers in terms of the algebra of the Stone-\v{C}ech compactification of discrete semigroups. The general case was proved by H. Shi and H. Yang. Using the notions of ultrafilters,  J. McLeod introduced the notions of Central sets for commutative adequate partial semigroups, however for noncommutative cases, Central sets can be defined similarly. In this article, introducing the notions of topological dynamics for partial semigroup actions, we find an equivalent dynamical characterization of central sets in partial semigroups\footnote{Recently in \cite{GTG}, authors attempted to do the same but in a different approach.}. Throughout our article, we follow the approach of N. Hindman and D. Strauss \cite{HS1}.
	\end{abstract}

	

	\section{Introduction}
For any nonempty set $X$, let $\mathcal{P}_f(X)$ be the set of all nonempty finite subsets of $X.$
Arithmetic Ramsey theory deals with the monochromatic patterns found in any given finite coloring of the
integers or of the natural numbers $\mathbb{N}$. Here ``coloring” means disjoint partition and a set is called ``monochromatic” if it is included in one piece of the partition. A cornerstone result in this field of
research is Van der Waerden's Theorem \cite{14}, which states that for any finite coloring of the natural numbers one always finds arbitrarily long monochromatic arithmetic progressions. For any injective sequence $\langle x_n\rangle_n$, a set of the form $FS(\langle x_n\rangle_n)=\left\lbrace \sum_{t\in H}x_t:H\text{ is a nonempty finite subset of }\mathbb{N}\right\rbrace$ is called an $IP$ set. Hindman's finite sum theorem  \cite{H74} is one of the fundamental theorems in Ramsey theory which says that 
 for every finite coloring of $\mathbb{N}$, there exists a monochromatic $IP$ set.
Central sets play a major role in the study of Arithmetic Ramsey theory. Basically many Ramsey theoretic patterns are found in Central sets. To recall the notions of Central sets, we need the following notions of dynamical systems. For details, literature on Central sets we refer to the article \cite{H20}.

\subsection{Preliminaries}
\subsubsection{Topological dynamics}
Before we start our discussion, we need to recall some basic ideas from topological dynamics.
	\begin{definition}[Dynamical system]\label{dynamical system}
	Let $S$ be a semigroup. A dynamical system is a pair $\left(X,\langle T_s\rangle_{s\in S}\right)$ such that 
	\begin{itemize}
		\item[(i)] $X$ is a compact Hausdroff space,
		\item[(ii)]for each $s$, $T_{s}:X\rightarrow X$  is continuous, and
		\item[(iii)] for all $s,t\in S$, $T_{s} T_{t}=T_{st}$.
	\end{itemize}
	\end{definition}

 The notions of syndetic sets are intimately related to the study of topological dynamics and are important for their combinatorial applications. A set is said to be a syndetic set if a finite translation of this set covers the entire semigroup.
	\begin{definition}[Syndetic set]{\cite[Definition 4.38]{HS1}}
     	Let $S$ be a semigroup and let $A\subseteq S$. The set $A$ is syndetic if and only if there is some $H\in \mathcal{P}_{f}\left(S\right)$ such that $S=\bigcup_{s\in H}s^{-1}A$.
	\end{definition}
	
The notion of proximal points and uniformly recurrent points naturally arise in topological dynamics and are the most studied topic in topological dynamics.  Both of these sets have beautiful applications in combinatorics. For details, readers can see \cite{F}.
	\begin{definition}\cite[Definition 1.2(b),(c)]{BuH}
		Let $S$ be a semigroup and let $\left(X,\langle T_s\rangle_{s\in S}\right)$ be a dynamical system.
		
		\begin{itemize}
		\item[(i)] A point $y\in S$ is uniformly recurrent if and only if for every neighborhood $U$ of $y$, $\left\{s\in S:T_{s}\left(y\right)\in U\right\} $ is syndetic.
		\item[(ii)] The points $x$ and $y$ of $X$ are proximal if and only if for every neighborhood $U$ of the diagonal in $X\times X$, there is some $s\in S$ such that $\left(T_{s}\left(x\right),T_{s}\left(y\right)\right)\in U$.
		\end{itemize}
	\end{definition}
In \cite{F}, H. Furstenberg introduced the notions of Central sets using topological dynamics. Before we dive into the study of central sets, let us recall the following notions of dynamically central sets from \cite{HS1}.	
	\begin{definition}[Dynamically central sets]\label{key}{\cite[Definition 19.20 Page-499]{HS1}}
		Let $S$ be a semigroup. A set $C\subseteq S$ is dynamically central if and only if there exists a dynamical system $\left(X,\langle T_s\rangle_{s\in S}\right)$, points $x$ and $y$ in $X$, and a neighborhood $U$ of $y$ such that
		\begin{itemize}
		\item[(i)] $y$ is a uniformly recurrent point of $X$,
		\item[(ii)] $x$ and $y$ are proximal, and 
		\item[(iii)] $C=\left\{s\in S:T_{s}\left(x\right)\in U\right\}$.
		
		\end{itemize}

	\end{definition}
 As later in \cite{BH} and \cite{SY}, authors found an equivalent definition of Central sets using the theory of ultrafilters, to discuss their results we need to recall the theory of the Algebra of the Stone-\v{C}ech compactification of discrete semigroups.
\subsubsection{Algebra of the Stone-\v{C}ech compactification of a semigroup }
	Let $\left(S,\cdot\right)$ be a discrete semigroup, and let $\beta S$ be the Stone-\v{C}ech compactification of $(S,\cdot ).$ For any $(\neq \emptyset)\,A\subseteq S,$ let $\{\overline{A}=\{p\in \beta S: A\in p\}.$
	The set $\{\overline{A}:(\neq \emptyset)\,A\subseteq S\}$ forms a basis. One can show that with this topology $(\beta S, \cdot)$ is a compact Hausdorff space.
 The operation `$\cdot$' on $S$ can be extended over $\beta S$ in the following way:
 for any $p,q\in \beta S$, $p\cdot q\in \beta S$ is defined as $$A\in p\cdot q\iff \{x:\{y:x\cdot y\in A\}\in q\}\in p.$$
 
 One can show that $(\beta S,\cdot)$ is a compact right topological semigroup (meaning that for each $p\in\beta S$ the function $\rho_{p}\left(q\right):\beta S\rightarrow\beta S$ defined by $\rho_{p}\left(q\right)=q\cdot p$ 
	is continuous) with $S$ contained in its topological center (meaning
	that for any $x\in S$, the function $\lambda_{x}:\beta S\rightarrow\beta S$
	defined by $\lambda_{x}(q)=x\cdot q$ is continuous). This is a famous
	Theorem due to Ellis that if $S$ is a compact right topological semigroup
	then the set of idempotents $E\left(S\right)\neq\emptyset$. A nonempty
	subset $I$ of a semigroup $T$ is called a $\textit{left ideal}$
	of $S$ if $TI\subset I$, a $\textit{right ideal}$ if $IT\subset I$,
	and a $\textit{two sided ideal}$ (or simply an $\textit{ideal}$)
	if it is both a left and right ideal. A $\textit{minimal left ideal}$
	is the left ideal that does not contain any proper left ideal. One can prove that a set  $A$  is  syndetic if  for every left ideal $L$ of $\delta S$, $\overline{A}\cap L\neq\emptyset$.
 Similarly,
	we can define $\textit{minimal right ideal}$ and $\textit{smallest ideal}$.
	
	Any compact Hausdorff right topological semigroup $T$ has the smallest
	two sided ideal
	
	$$
	\begin{aligned}
		K(T) & =  \bigcup\{L:L\text{ is a minimal left ideal of }T\}\\
		&=  \bigcup\{R:R\text{ is a minimal right ideal of }T\}.
	\end{aligned}$$

	Given a minimal left ideal $L$ and a minimal right ideal $R$, $L\cap R$
	is a group, and in particular contains an idempotent. If $p$ and
	$q$ are idempotents in $T$ we write $p\leq q$ if and only if $pq=qp=p$.
	An idempotent is minimal with respect to this relation if and only
	if it is a member of the smallest ideal $K(T)$ of $T$. Given $p,q\in\beta S$
	and $A\subseteq S$, $A\in p\cdot q$ if and only if the set $\{x\in S:x^{-1}A\in q\}\in p$,
	where $x^{-1}A=\{y\in S:x\cdot y\in A\}$. See \cite{HS} for
	an elementary introduction to the algebra of $\beta S$ and for any
	unfamiliar details.
	
	\subsubsection*{Central sets}
 In topological semigroup theory it's smallest two sided ideal enjoys beautiful properties. And using those properties one can derive many monochromatic patterns. The following notions of central sets is defined using the notions of ultrafilters.
	\begin{definition}\cite[Definition 4.42, Page-102]{HS1}
		Let $(S,\cdot)$ be  semigroup and  $A\subseteq S$. Then $A$ is said to be central if and only if there exists a minimal idempotent $p\in \beta S$ such that $A\in p$.
	\end{definition}
For the set of natural numbers the following theorem was proved in \cite{BH}, and for general semigroups in \cite{SY}. 	
	\begin{theorem}\textup{\cite[Theorem 19.27 Page-501]{HS1}}
	Let $(S,\cdot )$ be a semigroup and $B\subseteq S$. Then $B$ is central if and only if $B$ is dynamically central.
	\end{theorem}

 \subsubsection{Partial semigroup}
	A partial semigroup is a pair $\left(S,\cdot\right),$ where `$\cdot$' maps a subset of $S\times S$ to $S$ and for all $a,b,c\in S$, $\left(a\cdot b\right)\cdot c=a\cdot\left(b\cdot c\right)$ in the sense that if either side is defined, then so is the other and they are equal.
	\begin{ex}\label{e}\textcolor{white}{}
	\begin{itemize}
	\item[(1)] Let $\mathcal{R} = \left\{A: \text{ there exists } m,n\in \mathbb{N}  \text{ such that } A  \text{ is an }  m\times \text{ matrix with entries from }  \mathbb{Z}\right\}$, with the usual matrix multiplication. We know that for an $m\times n$ matrix $M$ and an $m^{\prime}\times n^{\prime}$ matrix $N$ in $\mathcal{R}$, $M\cdot N$ is defined if and only if $n=m^{\prime}$. so if we define `$\cdot$' as follows: 
	$$ M=
	\begin{cases}
		M\cdot N \text{ if } n=m^{\prime}\\
		\text{undefined otherwise}
	\end{cases}$$ then $\left(\mathcal{R},\cdot\right)$ is a partial semigroup.
	
	\item[(2)] Given a sequence $\langle x_{n}\rangle_{n=1}^{\infty}$ in some semigroup $\left(S,\cdot\right)$, let $T=FP\left(\langle x_{n}\rangle_{n=1}^{\infty}\right)$, where $$FP\left(\langle x_{n}\rangle_{n=1}^{\infty}\right)=\left\{\prod_{n\in F}x_{n}:F\in \mathcal{P}_{f}\left(\mathbb{N}\right)\right\}$$ and product are taken in increasing order of indices. then $T$ is not likely to be closed under $\cdot$. on the other hand, if we take,  $$ \left(\prod_{n\in F}x_{n}\right)\cdot\left(\prod_{n\in G}x_{n}\right)=
	\begin{cases}
		\prod_{n\in{F\cup G}}x_{n} \text{ if } \max{F}<\min{G}\\
		\text{ undefined } \text{ if } \max{F}\geq \min{G}
	\end{cases}$$ Then $\left(T,\cdot\right)$ is a partial semigroup.
				
	\end{itemize}
 \end{ex}
The study of ``partial semigroup'' has an important role in Ramsey theory. For details see \cite{HFM}. For combinatorial application, we rely on the study of some special class of partial semigroups: ``adequate partial semigroups''.
	\begin{definition}[Adequate partial semigroup]
		Let $\left(S,\cdot\right)$ be a partial semigroup.
		\begin{itemize}
			\item[(i)] For $s\in S$, $\phi\left(t\in S:s\cdot t \text{ is defined }\right)$.
			
			\item[(ii)] For $H\in \mathcal{P}_{f}\left(S\right)$,  $\sigma\left(H\right)=\cap_{s\in H}\phi \left(s\right)$
			
			\item[(iii)] $\left(S,\cdot\right)$ is adequate if and only if $\sigma\left(H\right)\neq \emptyset$ for all $H\in \mathcal{P}_{f}\left(S\right)$.
		\end{itemize}
		
	\end{definition}
 In Example \ref{e}, one can verify that unlike $\left(\mathcal{R},\cdot\right)$, the partial semigroup $\left(T,\cdot \right)$ is adequate. In the case of $\left(\mathcal{R},\cdot\right)$, notice that for any $H\in \mathcal{P}_{f}\left(\mathcal{R}\right)$, $\sigma\left(H\right)\neq \emptyset$ if and only if $\mathcal{H}=\left\{A:A \text{ is a matrix with } r \text{ columns for some fixed } r\in\mathbb{N} \right\}$.
 
\subsubsection{Algebra of the Stone-\v{C}ech compactification of discrete partial semigroups}

Note that if $(S,\cdot)$ is a partial semigroup, and $\beta S$ is the set of all ultrafilters over $S,$ then $(\beta S,\cdot)$ is also a partial semigroup.
 One of the main advantage of ``adequate partial semigroup'' is that we can find a semigroup structure inside $\beta S$, which is not possible for general partial semigroups.
 \begin{definition}
 	Let $\left(S,\cdot\right)$ be a partial semigroup. Then $$\delta S=\bigcap_{x\in S}\overline{\phi\left(x\right)}=\bigcap_{H\in \mathcal{P}_{f}\left(S\right)}\overline{\sigma\left(H\right)}.$$
 \end{definition}
 
Clearly $\delta S\subseteq \beta S$. Most interesting fact is that this set $\delta S$ along with the operation `$\star$' forms a semigroup. For $x\in S$ and $A\subseteq S$, define $x^{-1}A=\left\{y\in \phi\left(x\right):x\cdot y\in A\right\}.$
 The following lemma describes the algebraic structure of  ``adequate partial semigroups''.
 
  \begin{lemma}\textup{\cite[Lemma 2.4]{HM}}
 	Let $S$ be an adequate partial semigroup.
 	\begin{itemize}
 		\item[(a)] Let $x\in S$, let $q\in \overline{\phi\left(x\right)}$, and let $A \subseteq S$. Then $$ A\in x\cdot q\iff x^{-1}A\in q,$$
 		\item[(b)] Let $p\in \beta S$, let $q\in \delta S$, and let $A\subseteq S$. Then $$A\in p\cdot q\iff \left\{x\in S:x^{-1}A\in q\right\}\in p.$$
 	\end{itemize}
 	
 \end{lemma}
 The following theorem says that $\delta S$ is a compact Hausdorff right topological semigroup.
 \begin{theorem}\textup{\cite[Theorem 2.6]{HM}}
 	Let $S$ be the adequate partial semigroup. Then the restriction of the operation given in Theorem \ref{Operation}, $\delta S$ is a compact right topological semigroup.
 \end{theorem}
 
 The fact that $\delta S$ is a compact right topological semigroup provides a natural context for the notion of ``central'' sets  in an adequate partial semigroup.
 \begin{definition}\label{a}
 	Let $S$ be an adequate partial semigroup and let $A\subseteq S$. Then $A$ is central if and only if there is some minimal idempotent $p\in \delta S$ such that $A\in p$.
 \end{definition}
 
\subsection*{Structure of the paper:} Introducing the notions of topological dynamics for partial semigroup action, we define dynamical central sets in partial semigroups. Then we prove the notion is equivalent with the Definition \ref{a}. The same question was addressed in \cite{GTG}, but their approach is different and complicated.


	\section{Dynamical characterization of central sets}
	
Before we proceed let us introduce the notions of a dynamical system for the action of an adequate partial semigroup. This generalizes Definition \ref{dynamical system}. One may ask about the existence of such dynamical system, which they can find in the Lemma \ref{shift of sequence}.
		\begin{definition}\label{dynamical system for partial }
		Let $S$ be an adequate partial semigroup. A dynamical system is a pair $\left(X,\langle T_s\rangle_{s\in S}\right)$ such that 
		\begin{itemize}
			\item[$\left(i\right)$] $X$ is a compact Hausdroff space,
			\item[(ii)]For each $s$, $T_{s}:X\rightarrow X$ and $T_{s}$ is continuous, and
			\item[(iii)] For all $s,t\in S$, $T_{s} T_{t}=T_{st}$ if $t\in \phi\left(s\right)$.
		\end{itemize}
	\end{definition}

In partial semigroup, we can define syndetic sets in two different ways. The first one is in the sense of a combinatorial approach, known as $\check{c}$-syndetic sets. 
	\begin{definition}{\cite[Definition 2.13]{HM}}\label{syndetic in partial}
		Let $\left(S,\cdot\right)$ be an adequate partial semigroup and let $A\subseteq S$. The set $A$ is $\check{c}$-syndetic if and only if there is some $H\in \mathcal{P}_{f}\left(S\right)$ such that $\sigma\left(H\right)\subseteq\bigcup_{t\in H}t^{-1}A$.
	\end{definition}
The following lemma gives a characterization of	$\check{c}$-syndetic sets.

	\begin{lemma}\textup{\cite[Lemma 2.14]{HM}}
		Let $\left(S,\cdot\right)$ be an adequate partial semigroup and let $A\subseteq S$. Then $A$ is $\check{c}$-syndetic  if and only if there exists $H\in \mathcal{P}_{f}\left(S\right)$ such that $\delta S\subseteq\bigcup_{t\in H}\overline{t^{-1}A}$.
	\end{lemma}
The following notions of syndetic set can be defined using the structure of $\delta S.$
\begin{definition}\label{ algebraic syndetic in partial}
		Let $S$ be an adequate partial semigroup and let $A\subseteq S$. Then $A$  is  syndetic if  for every left ideal $L$ of $\delta S$, $\overline{A}\cap L\neq\emptyset$.
	\end{definition}
However one can prove that both of the above notions are not same.
\begin{theorem}{\cite[Theorem 3.4]{M1}}
    There exists an adequate partial semigroup $\left(T,\cdot\right)$ and a $\check{c}$-syndetic subset $A$ of $T$ which is not syndetic.
\end{theorem}

 For our purpose, we need the following variations of syndetic sets.
	\begin{definition}[$\check{a}$-syndetic set]\label{new syndetic set}
		Let $S$ be an adequate partial semigroup and let $A\subseteq S$. then $A$  is said to be  $\check{a}$-syndetic if for every $G\in \mathcal{P}_{f}\left(S\right)$, there exists $H\in\mathcal{P}_{f}\left(\sigma\left(G\right)\right)$ such that $\sigma\left(H\cup K\right)\subseteq \bigcup_{t\in H}{t^{-1}A}$ for some $K\in \mathcal{P}_{f}\left(S\right)$.
	\end{definition}
The following theorem gives us the combinatorial characterization of syndetic sets.
 
	\begin{theorem}
		Let $S$ be an adequate partial semigroup and let $A\subseteq S$. Then $A$  is  $\check{a}$-syndetic if  and only if for every left ideal $L$ of $\delta S$, $\overline{A}\cap L\neq\emptyset$.

	\end{theorem}
	
	\begin{proof}
		Let $A\subseteq S$ is syndetic and $L$ is any left ideal $L$ of $\delta S$. Let $p\in L$. For each $G\in \mathcal{P}_{f}\left(S\right)$, pick  $H\in\mathcal{P}_{f}\left(\sigma\left(G\right)\right)$ such that  $\sigma\left(H\cup K\right)\subseteq \bigcup_{t\in H}{t^{-1}A}$ for some $K\in \mathcal{P}_{f}\left(S\right)$. Then there exists $t_{G}\in \sigma\left(G\right)$ such that $t_{G}^{-1}A\in p$. Let $$E=\left\{t_{G}:G\in \mathcal{P}_{f}\left(S\right)\right\}\cup\left\{\sigma\left(G\right):G\in \mathcal{P}_{f}\left(S\right)\right\}$$ has finite intersection property. So there exist $q\in \beta S$ such that $E\subseteq q$. As $\sigma\left(G\right)\in q$ for all $G\in \mathcal{P}_{f}\left(S\right)$, then $q\in \delta S$. Now $ \left\{t_{G}:G\in \mathcal{P}_{f}\left(S\right)\right\}\in q$. As $$ \left\{t_{G}:G\in \mathcal{P}_{f}\left(S\right)\right\}\subseteq \left\{t\in S:t^{-1}A\in p\right\},$$ so $\left\{t\in S:t^{-1}A\in p\right\}\in q$. Thus $A\in qp$. Also $qp\in L$, as $p\in L$ and $L$ is a left ideal of $\delta S$. So $qp\in\overline{A}\cap L\neq\emptyset$.\\

		\noindent\textbf{Conversely} assume that $\overline{A}\cap L\neq\emptyset$ for every left ideal $L$ of $\delta S$. Suppose that $A$ is not syndetic. So, there exists $G\in \mathcal{P}_{f}\left(S\right)$, for all $H\in\mathcal{P}_{f}\left(\sigma\left(G\right)\right)$ such that $\sigma\left(H\cup K\right)\setminus \bigcup_{t\in H}{t^{-1}A}\neq\emptyset$ for every $K\in \mathcal{P}_{f}\left(S\right)$. Then the set $$E=\left\{\sigma\left(H\right)\setminus\bigcup_{t\in H}t^{-1}A:H\in\mathcal{P}_{f}\left(\sigma\left(G\right)\right)\right\}\bigcup\left\{\sigma\left(K\right):K\in\mathcal{P}_{f}\left(S\right)\right\}$$ has finite intersection property. Now pick $p\in \delta S$ such that $E\subseteq p$. As $Lp$ is a left ideal of $\delta S$ and $A$ is algebraically syndetic, $\overline{A}\cap Lp\neq\emptyset$, then pick $q\in L$ such that $A\in qp$, which implies $\left\{x\in S:x^{-1}A\in p\right\}\in q$. Also $\sigma\left(G\right)\in q$ for all  because of $q\in \delta S$. So there exist $$x\in \left\{x\in S:x^{-1}A\in p\right\}\cap\sigma\left(G\right)\in q$$ such that $x^{-1}A\in p$ with $x\in \sigma\left(G\right)$, which contradicts the fact that  $\sigma\left(\left\{x\right\}\right)\setminus x^{-1}A\in p$
		
	\end{proof}
Now we introduce the notions of  uniformly recurrent and proximal points for adequate partial semigroup actions.
	\begin{definition}\label{uniform and proximality}
		Let $(S,\cdot)$ be an adequate partial semigroup and let $\left(X,\langle T_{s}\rangle_{s\in S}\right)$ be a dynamical system.
		
		\begin{itemize}
			\item[(i)] A point $y\in S$ is uniformly recurrent if and only if for every neighborhood $U$ of $y$, $\left\{s\in S:T_{s}\left(y\right)\in U\right\} $ is syndetic.
			\item[(ii)]The points $x$ and $y$ of $X$ are proximal if and only if for every neighborhood $U$ of the diagonal in $X\times X$ and for each $H\in \mathcal{P}_{f}\left(S\right)$ there is some $s\in \sigma\left(H\right)$ such that $\left(T_{s}\left(x\right),T_{s}\left(y\right)\right)\in U$.
		\end{itemize}
	\end{definition}

 Let us introduce the notions of dynamical central sets in partial semigroups, which is an analogous version of Definition \ref{key}.
	\begin{definition}\label{dynamically central sets in partial }
		Let $(S,\cdot)$ be an adequate partial  semigroup. A set $C\subseteq S$ is dynamically central if and only if there exist a dynamical system $\left(X,\langle T_{s}\rangle_{s\in S}\right)$ , points $x$ and $y$ in $X$, and a neighborhood $U$ of $y$ such that
		\begin{itemize}
			\item[(1)] $y$ is a uniformly recurrent point of $X$,
			\item[(2)] $x$ and $y$ are proximal, and 
			\item[(3)] $C=\left\{s\in S:T_{s}\left(x\right)\in U\right\}$.
			
		\end{itemize}
			
	\end{definition}
To prove the Definitions \ref{ algebraic syndetic in partial} and \ref{dynamically central sets in partial } are equivalent, we use the concept of taking limit along ultrafilters.	
	\begin{definition}{\cite[Definition 3.44, Page-74]{HS1}}\label{P-limit}
		Let $D$ be a discrete space, let $p\in \beta D$, let $\langle x_{s}\rangle_{s\in D}$ be an indexed family in a topological space $X$, and let $y\in X$. Then $p\text{-}{\lim_{s\in D}} x_{s}=y$ if and only if for every neighborhood $U$ of $y$, $\left\{s\in D:x_{s}\in U\right\}\in p$. 
	\end{definition}

 The following theorem is essential for us.
	\begin{theorem}\textup{\cite[Theorem 3.48, Page-75]{HS1}}
		Let $D$ be a discrete space, let $p\in \beta D$, and let $\langle x_{s}\rangle_{s\in D}$ be an index family in a  topological space $X$.
		\begin{itemize}
		\item[(a)]If $p\text{-}{\lim_{s\in D}} x_{s}$ exists, then it is unique.
		\item[(b)] If $X$ is a compact space, then $p\text{-}{\lim_{s\in D}} x_{s}$ exists.
		\end{itemize}
	\end{theorem}
	The following theorem gives us the characterization of proximal points in terms of ultrafilters.
	
	\begin{lemma}
		Let $S$ be an adequate partial semigroup. Let $\left(X,\langle T_{s}\rangle_{s\in S}\right)$ be a dynamical system and let $x,y\in X$. Then $x$ and $y$  are proximal if and only if there is some $p\in \delta S$ such that $T_{p}\left(x\right)=T_{p}\left(y\right)$.
	\end{lemma}
	
	\begin{proof}
		Let $x$ and $y$ are proximal. Let $\mathcal{N}$ be the set of all neighborhoods of the diagonal in $X\times X$. For each $U\in \mathcal{N}$, let $B_{U}=\left\{s\in S:\left(T_{s}\left(x\right),T_{s}\left(y\right)\right)\in U\right\}$. From the Definition \ref{uniform and proximality} of proximality, $$\mathcal{F}=\left\{B_{U}:U\in \mathcal{N}\right\}\cup \left\{\sigma\left(H\right):H\in \mathcal{P}_{f}\left(S\right)\right\}$$ has finite intersection property. Now choose $p\in \beta S$ such that $\mathcal{F}\subseteq p$. As $\left\{\sigma\left(H\right):H\in \mathcal{P}_{f}\left(S\right)\right\}\subseteq p$, we may consider $p\in \delta S$. Let $z=T_{p}\left(x\right)$. To see that $z=T_{p}\left(y\right)$, let $V$ be an open neighbourhood of $z$ in $X$. Note that $X\setminus V$ is a closed subset of $X$ such that $z\notin X\setminus V$. Since $X$ is a compact Hausdorff, there exists disjoint open sets $V_{1},V_{2}$ such that $z\in V_{1}$ and $X\setminus V\subseteq V_{2}$. Let $U=\left(V\times V\right)\cup \left(V_{2}\times X\right)$. Then $U$ is a neighborhood of the diagonal in $X\times X$ such that $\pi_{2}\left(\pi_{1}^{-1}\left(V_{1}\right)\cap U\right)\subseteq V$, where $pi_{1}$ and $pi_{2}$ denote the first and second projection of $X\times X$ onto $X$ respectively. Let $$E=\left\{s\in S:T_{s}\left(x\right)\in V_{1}\right\}$$ and let $$F=\left\{s\in S:\left(T_{s}\left(x\right),T_{s}\left(y\right)\right)\in U\right\}.$$ Then $E,F\in p$ and $$E\cap F\subseteq \left\{s\in S:T_{s}\left(x\right)\in V\right\}.$$ Thus $\left\{s\in S:T_{s}\left(x\right)\in V\right\}\in p$ for every open neighborhood $V$ of $z$.
		
	\textbf{Conversely} suppose we have $p\in \delta S$ such that $T_{p}\left(x\right)=T_{p}\left(y\right)=z$. Let $U$ be a neighborhood  of the diagonal in $X\times X$. Choose an open neighborhood $V$ of $z$ in $X$ such that $V\times V\subseteq U$. Let $$B=\left\{s\in S:T_{s}\left(x\right)\in V\right\}\text{ and }C= \left\{s\in S:T_{s}\left(y\right)\in V\right\}.$$ Then $B\cap C\in p$. For each $H\in\mathcal{F}_{f}\left(S\right)$, choose $s\in \sigma\left(H\right)\cap B\cap C$ such that $\left(T_{s}\left(x\right),T_{s}\left(y\right)\right)\in V\times V\subseteq U$.
	
	\end{proof}
	
	We need the following structure theorem from \cite{HFM}.
	
	\begin{theorem}\label{Operation}\textup{\cite[Theorem 2.2]{HFM}}
		Let $\left(S.\cdot\right)$ be an adequate partial semigroup. Let $$D=\left(\bigcap_{x\in S}\left(\left\{x\right\}\times \overline{\phi\left(x\right)}\right)\cup \left(\beta S\times \delta S\right)\right).$$ Then the operation $\cdot$ can be extended uniquely to $D$ so that
		\begin{itemize}
			\item[(a)] for each $x\in S$, the function $\lambda_{x}:\overline{\phi\left(x\right)}\rightarrow\beta S$, defined by $\lambda_{x}\left(q\right)=x\cdot q$, is continuous, and
			\item[(b)] for each $q\in \delta S$, the function $\rho_{q}:\beta S\rightarrow \beta S$, defined by $\rho_{q}\left(q\right)=p\cdot q$ is continuous.
		\end{itemize}
		
	\end{theorem}
	
The following theorem is a version of van der corput lemma, a very useful result in ergodic Ramsey theory.
	
\begin{theorem}\label{vd}\textup{\cite[Theorem 4.5 Page-87]{HS1}}
	Let $\left(S,\cdot\right)$ be a semigroup, $X$ be a topological space, $\langle x_{s}\rangle_{s\in S}$  be an index family in $X$, and $p,q\in \beta S$. If all limits involved exists, then $$\left(pq\right)\text{-}\lim_{v\in S}x_{v}=p\text{-}\lim_{s\in S}q\text{-}\lim_{t\in S}x_{st}.$$
\end{theorem}
	It is obvious from  the Theorem\ref{Operation} (a), that for $q\in \delta S$, $q\text{-}\lim_{t\in \phi\left(s\right)}{st}=sq$, together with   Theorem\ref{Operation} (b) imply the following:
	\begin{lemma}
		Let $\left(S.\cdot\right)$ be an adequate partial semigroup.If $q\in \delta S$ and $p\in \beta S$, then $$\left(pq\right)=p\text{-}\lim_{s\in S}q\text{-}\lim_{t\in \phi\left(s\right)}{st}.$$
			
	\end{lemma}	
The following theorem is a version of Theorem \ref{vd} for adequate partial semigroups.
\begin{theorem}
	Let $\left(S,\cdot\right)$ be an adequate partial semigroup, let $X$ be a compact Housdroff space and $q\in \delta S$ and $p\in \beta S$. Let $\langle x_{n}\rangle_{s\in S}$  be an index family in $X$, and $p,q\beta S$. Then $$\left(pq\right)\text{-}\lim_{v\in S}x_{v}=p\text{-}\lim_{s\in S}q\text{-}\lim_{t\in \phi\left(s\right)}x_{st}.$$	
\end{theorem}

\begin{proof}
	Let $f:\rightarrow X$ be defined by $f\left(s\right)=x_{s}$ and let $\widetilde{f}$ be the continuous extension of $f$ to $\beta S$. Then one has  
		$$
	\begin{aligned}
		\left(pq\right)\text{-}\lim_{v\in S}x_{v} & =  \left(pq\right)\text{-}\lim_{v\in S} \widetilde{f}\left(v\right)\\
		&=  \widetilde{f}\left(\left(pq\right)\text{-}\lim_{v\in S}x_{v}\right)\\
		&= \widetilde{f}\left(pq\right)\\
		&= \widetilde{f}\left(p\text{-}\lim_{s\in S}q\text{-}\lim_{t\in \phi\left(s\right)}{st}\right)\\
		&= p\text{-}\lim_{s\in S}q\text{-}\lim_{t\in \phi\left(s\right)}\widetilde{f}\left(st\right)\\
		&= p\text{-}\lim_{s\in S}q\text{-}\lim_{t\in \phi\left(s\right)}x_{st}.
	\end{aligned}$$
\end{proof}

We need the following corollary of the above theorem in our proof.
\begin{corollary}\label{action of delta s}
Let $S$ be an adequate partial semigroup. Let $\left(X,\langle T_{s}\rangle_{s\in S}\right)$ be a dynamical system. Let $p\in \beta S$ and $q\in \delta S$, then for any $x\in X$, $T_{pq}\left(x\right)=T_{p}\left(T_{q}\left(x\right)\right)$.
\end{corollary}

\begin{proof}
	Taking $f\left(s\right)=T_{s}\left(x\right)$, 
	$$
	\begin{aligned}
		T_{pq}\left(x\right) &= \left(pq\right)\text{-}\lim_{v\in S}T_{v}\left(x\right)\\
		&= p\text{-}\lim_{s\in S}q\text{-}\lim_{t\in \phi\left(s\right)}T_{st}\left(x\right)\\
		&= p\text{-}\lim_{s\in S}q\text{-}\lim_{t\in \phi\left(s\right)}T_{s}T_{t}\left(x\right)\\
		&= p\text{-}\lim_{s\in S}T_{s}\left(q\text{-}\lim_{t\in \phi\left(s\right)}T_{t}\left(x\right)\right)\\
		&= p\text{-}\lim_{s\in S}T_{s}\left(T_{q}\left(x\right)\right)\\
		&= T_{p}\left(T_{q}\left(x\right)\right).
	\end{aligned} $$

\end{proof}
The following lemma connects uniformly recurrent points and minimal left ideals.	
	\begin{lemma}
		Let $S$ be an adequate partial semigroup. Let $\left(X,\langle T_{s}\rangle_{s\in S}\right)$ be a dynamical system and $L$ be a minimal left ideal of $\delta S$ and $x\in X$. The following statements are equivalent:
		\begin{itemize}
			\item[(a)] The point $x$ is a uniformly recurrent point of $\left(X,\langle T_{s}\rangle_{s\in S}\right)$.
			\item[(b)] There exists $p\in L$ such that $T_{p}\left(x\right)=x$.	
			\item[(c)] There exists an idempotent $p\in L$ such that $T_{p}\left(x\right)=x$.
		\end{itemize}
	\end{lemma}
	
	\begin{proof}
		(a)$\implies$(b) Choose any $q\in L$. Let $\mathcal{N}$ be the set of neighborhoods of $x$ in $X$. For each $U\in \mathcal{N}$, let $$B_{U}=\left\{s\in S:T_{s}\left(x\right)\in U\right\}.$$ Sine $x$ is a  uniformly recurrent point, each  $B_{U}$ is syndetic, so for  For each $G\in \mathcal{P}_{f}\left(S\right)$, pick  $H_{U,G}\in\mathcal{P}_{f}\left(\sigma\left(G\right)\right)$ such that  $\sigma\left(H\cup K\right)\subseteq \bigcup_{t\in H_{U,G}}{t^{-1}A}$ for some $K\in \mathcal{P}_{f}\left(S\right)$. So for each $U\in \mathcal{N}$ and each $G\in \mathcal{P}_{f}\left(S\right)$ pick $t_{U,G}\in H_{U,G}$, such that $t_{U,G}^{-1}B_{U}\in q$. Given $U\in \mathcal{N}$, let $$C_{U}=\left\{t_{V,G}:V\subseteq U \text{ and }G\in \mathcal{P}_{f}\left(S\right)\right\}.$$ Then $$\left\{C_{U}:U\in\mathcal{N}\right\}\cup\left\{G:G\in\mathcal{P}_{f}\left(S\right)\right\}$$ has the finite intersection property. so, pick $r\in \delta S$ such that $\left\{C_{U}:U\in\mathcal{N}\right\}\subseteq r$ and let $p=rq$. Since $L$ is a left ideal of $\delta S$, $p\in L$.  To prove $T_{p}\left(x\right)=x$, it is sufficient to show that $C_{U}\subseteq \left\{t\in S:t^{-1}B_{U}\in q\right\}$, where $U\in \mathcal{N}$ with $B_{U}\in p$. So, let $t\in C_{U}$ and pick $V\in\mathcal{N}$ and $G\in \mathcal{P}_{f}\left(S\right)$ such that $V\subseteq U$ and $t=t_{V,G}$. Then $t^{-1}B_{V}\in q$ and $t^{-1}B_{V}\subseteq t^{-1}B_{U}$.
		
		(b)$\implies$(c) $K=\left\{p\in L:T_{p}\left(x\right)=x\right\}$. $K$ is non-empty and it is a routine exercise that $K$ is a compact subsemigroup of $L$.
			
		(c)$\implies$(a) Let $U$ be a neighborhood of $x$ and let $B=\left\{s\in S:T_{s}\left(x\right)\in U\right\}$ and suppose that $B$ is not syndetic. So, there exists $G\in \mathcal{P}_{f}\left(S\right)$,  such that the set $$\mathcal{E}=\left\{\sigma\left(H\right)\setminus\bigcup_{t\in H}t^{-1}A:H\in\mathcal{P}_{f}\left(\sigma\left(G\right)\right)\right\}\bigcup\left\{\sigma\left(K\right):K\in\mathcal{P}_{f}\left(S\right)\right\}$$ has finite intersection property. So, pick $r\in \beta S$ such that $\mathcal{E}\subseteq r$. As $\left\{\sigma\left(K\right):K\in\mathcal{P}_{f}\left(S\right)\right\}\subseteq r$, $r\in \delta S$. Then $\delta Sr\cap\overline{B}=\emptyset $	( For suppose instead one had some $q\in\delta S$ with $B\in qr$, then $t^{-1}B\in r$ for $t\in \mathcal{P}_{f}\left(\sigma\left(G\right)\right)$). Now $\delta Srp$ is a left ideal of $\delta S$ which is contained in $L$, as $L$ is the minimal left ideal of $\delta S$, then $\delta Srp=L$. Thus we may pick some $q\in \delta Sr$ such that $qp=p$ . $T_{q}\left(x\right)=T_{q}\left(T_{p}\left(x\right)\right)=T_{qp}\left(x\right)=T_{p}\left(x\right)=x$, so in particular $B\in q$. But, $q\in \delta Sr$, contradicts the fact that $\delta Sr\cap\overline{B}=\emptyset$
	\end{proof}
	The following theorem gives a sufficient condition for a uniformly recurrent point to be proximal with a given point.
	\begin{theorem}
		Let $S$ be an adequate partial semigroup. Let $\left(X,\langle T_{s}\rangle_{s\in S}\right)$ be a dynamical system and let $x\in X$. Then there is a uniformly recurrent point $y\in X$ such $x$ and $y$ are proximal.
	\end{theorem}
	
	\begin{proof}
		Let $L$ be any minimal left ideal of $\delta S$ and pick an idempotent $p\in L$. Let $y=T_{p}\left(x\right)$. Now, $T_{p}\left(y\right)=T_{p}\left(T_{p}\left(x\right)\right)=T_{pp}\left(x\right)=T_{p}\left(x\right)=y$. Then $x$ and $y$ are proximal and $y$ is uniformly recurrent point.
	\end{proof}
	
The following theorem gives a necessary condition for two points to be proximal.
		\begin{theorem}\label{eta}
		Let $S$ be an adequate partial semigroup. Let $\left(X,\langle T_{s}\rangle_{s\in S}\right)$ be a dynamical system and let $x,y\in X$. If $x$ and $y$ are proximal, then there is a minimal left ideal $L$ of $\delta S$ such that $T_{p}\left(x\right)=T_{p}\left(y\right)$ for all $p\in L$.
	\end{theorem}
	
	\begin{proof}
		$\left\{p\in\delta S:T_{p}\left(x\right)=T_{p}\left(y\right)\right\}\neq \emptyset$. It is a left ideal of $\delta S$, for every $p,q\in \delta S$, $T_{p}\left(x\right)=T_{p}\left(y\right)$ implies that $T_{qp}\left(x\right)=T_{q}\left(T_{p}\left(x\right)\right)=T_{q}\left(T_{p}\left(y\right)\right)=T_{qp}\left(y\right)$.
	\end{proof}
The following theorem connects minimal idempotents with proximal and uniformly recurrent points. 
	\begin{theorem}
		Let $S$ be an adequate partial semigroup. Let $\left(X,\langle T_{s}\rangle_{s\in S}\right)$ be a dynamical system and let $x,y\in X$. There is a minimal idempotent $p\in \delta S$ such that $T_{p}\left(x\right)=y$ if and only if $x$ and $y$ are proximal and $y$ are uniformly recurrent. 
	\end{theorem}
	
	\begin{proof}
	\textbf{Necessity:} Since $p$ is minimal, there is a minimal left ideal $L$ of $\delta S$ such that $p\in L$. Thus by the previous theorem, $y$ is uniformly recurrent, and $x$ and $y$ are proximal.\\

	\textbf{Sufficiency:} Pick by previous theorem a minimal left ideal $L$ of $\delta S$ such that $T_{p}\left(x\right)=T_{p}\left(y\right)$ for all $p\in $. Pick by Theorem \ref{eta}, an idempotent $p\in L$ such that $T_{p}\left(y\right)=y$.
		
	\end{proof}
Now, we provide a lemma, which is analog of \cite[Lemma 19.14, 495 ]{HS1}.

\begin{lemma}\label{shift of sequence}
  Let $S$ be an adequate partial semigroup. Let $X=\prod_{s\in S}\left\{0,1\right\}$ and for $s\in S$ define $T_{s}:X\rightarrow X$, by
   $$T_{s}\left(x\right)\left(t\right)=
 \begin{cases}

x\left(st\right) & \text{ if } t\in\phi\left(s\right)\\
0 & \text{ if } t\notin\phi\left(s\right)
\end{cases}.$$ Then  $\left(X,\langle T_{s}\rangle_{s\in S}\right)$ is a dynamical system.
\end{lemma}
\begin{proof}
Let $s\in S$. To see that $T_{s}$ is continuous it is sufficient two show that $\pi{t}T_{s}$ is continuous for each $t\in S$. Let $x\in X$, then $\left(\pi_{t}T_{s}\right)x=\pi_{t}\left(T_{s}x\right)=\left(T_{s}x\right)\left(t\right)$. Where, 

$$\left(T_{s}x\right)\left(t\right)=
 \begin{cases}

x\left(st\right)=\pi_{st}\left(x\right) & \text{ if } t\in\phi\left(s\right)\\
0 & \text{ if } t\notin\phi\left(s\right)
\end{cases}.$$ Therefore $T_{s}$ is continuous because of

$$\pi_{t}T_{s}=
 \begin{cases}

\pi_{st} & \text{ if } t\in\phi\left(s\right)\\
constant & \text{ if } t\notin\phi\left(s\right)
\end{cases}.$$
To prove the second part, let $s,t\in S$. Let $x\in X$ and  $T_{t}\left(x\right)\in X$ is defined by 

$$\left(T_{t}x\right)\left(u\right)=
 \begin{cases}

x\left(tu\right) & \text{ if } u\in\phi\left(t\right)\\
0 & \text{ if } u\notin\phi\left(t\right)
\end{cases}.$$ Then

$$\left(T_{s}T_{t}x\right)\left(u\right)=
 \begin{cases}

x\left(stu\right) & \text{ if } u\in\phi\left(t\right) \text{ and }tu\in\phi\left(s\right)\\
0 & \text{ if } u\notin\phi\left(t\right)\text{ or }tu\notin\phi\left(s\right)
\end{cases}.$$ Now, $u\in \phi\left(st\right)\iff u\in \phi\left(t\right)$ and $tu\in\phi\left(s\right)$. From these facts together with the above function, we have

$$\left(T_{s}T_{t}x\right)\left(u\right)=
 \begin{cases}

x\left(stu\right) & \text{ if } u\in\phi\left(st\right) \\
0 & \text{ if } u\notin\phi\left(st\right)
\end{cases}.$$ 
Therefor $T_{s}T_{t}=T_{st}$ if $t\in\phi\left(s\right)$
 \end{proof}
The following theorem is our main theorem which gives dynamical characterization of Central sets of adequate partial semigroups.

	\begin{theorem}
		 Let $S$ be an adequate partial semigroup and $B\subset S$. Then $B$ is central if and only if $B$ is dynamically central.
	\end{theorem}
 
\begin{proof}	
	\textbf{Necessity:} Let $G=S\cup \left\{e\right\}$, $X=\prod_{s\in G}\left\{0,1\right\}$ and for $s\in S$ define $T_{s}:X\rightarrow X$  by  $$T_{s}\left(x\right)\left(t\right)=
 \begin{cases}

x\left(st\right) & \text{ if } t\in\phi\left(s\right)\\
0 & \text{ if } t\notin\phi\left(s\right)
\end{cases}.$$ Then  $\left(X,\langle T_{s}\rangle_{s\in S}\right)$ is a dynamical system. Now let $x=1_{B}$, the characteristic function of $B$. Pick a minimal idempotent in $\delta S$ such that $B\in p$ and let $y=T_{p}\left(x\right)$. Then $y$ is a uniform recurrent point and $x$ and $y$ are proximal. Now let $U=\left\{z\in X:z\left(e\right)=y\left(e\right)\right\}$.
	  Then $U$ is a neighbourhood of $y$ in $X$. We note that $y\left(e\right)=1$. Indeed, $y=T_{p}\left(x\right)$, so, $\left\{s\in S:T_{s}\left(x\right)\in U\right\}\in p$
	   and we may choose some $s\in B$ such that $T_{s}\left(x\right)\in U$.
	    Then $y\left(e\right)=T_{s}\left(x\right)\left(e\right)=x\left(se\right)=x\left(s\right)=1$. Thus given any $s\in S$,
	$$
	\begin{aligned}
		s\in B &\iff x\left(s\right)=1\\
		&\iff T_{s}\left(x\right)\left(e\right)=1\\
		&\iff T_{s}\left(x\right)\in U.
	\end{aligned}$$
	
	\textbf{ Sufficiency:} Choose a dynamical system $\left(X,\langle T_{s}\rangle_{s\in S}\right)$, points $x,y\in X$ and neighbourhood $U$ of $y$ such that $x$ and $y$ are proximal  with $y$ uniformly recurrent point and $B=\left\{s\in S:T_{s}\left(x\right)\in U\right\}$. choose a minimal idempotent $p$ in $\delta S$ such that $T_{p}\left(x\right)=y$. then $B\in p$. This completes the proof.
\end{proof}

 \section*{Acknowledgement}  The second author of this paper is supported by NBHM postdoctoral fellowship with reference no: 0204/27/(27)/2023/R \& D-II/11927.
	
	\bibliographystyle{plain}

\end{document}